\documentclass[12pt]{amsart}

\usepackage{amssymb, amsmath, latexsym, a4,hyperref}
\usepackage[latin1]{inputenc}
\usepackage[all]{xy}

\usepackage{tikz}
\usepackage{pinlabel}
\usepackage{amsfonts}
\usepackage{amscd}
\usepackage{txfonts}
\usepackage{geometry}

\usetikzlibrary{matrix,calc,arrows}
\newcommand*\Z{\mathbb{Z}}

\newtheorem{De}{Definition}[section]
\newtheorem{Th}[De]{Theorem}
\newtheorem{Pro}[De]{Proposition}
\newtheorem{Le}[De]{Lemma}
\newtheorem{Co}[De]{Corollary}

\newtheorem{Ex}[De]{Example}

\numberwithin{equation}{subsection}

\def\xto#1{\xrightarrow[]{#1}}

\def\id{\sf Id}

\def\t{\otimes }

\def\et{{ \circledast}}
\def\ep{{ \vee}}

\def\Ec{\mathcal{E} }

\def\Z{\mathbb{Z} }

\def\mor{{\sf \Hom}_{\mathsf{El}}}

\def\maps{\mathsf{Maps}}

\def\1{^{-1}}

\def\aut{\mathsf{Aut}}
\def\aute{\mathsf{Aut_E}}

\def \Hom{\mathop{\sf Hom}\nolimits}

\def\Ell{\bf Ell_{GR}}

\begin{document}
	
\title{Elliptic groups and rings}

\author[I. Pirashvili]{Ilia Pirashvili}
\address{Institut f\"ur Mathematik\\
	Universit\"at Augsburg\\
	Universit\"atsstr. 2\\ 
	86159 Augsburg}
\email{\href{mailto:ilia.pirashvili@math.uni-augsburg.de}{ilia.pirashvili@math.uni-augsburg.de} \ \ \ Alternatively: \href{mailto:ilia\_p@ymail.com}{ilia\_p@ymail.com}}
\maketitle

\begin{abstract} As it is well known, one can define an abelian group on the points of an elliptic curve, using the so called chord-tangent law \cite{dale}, and a chosen point. However, that very chord-tangent law allows us to define a rather more obscure algebraic structure, which we call an elliptic group, on the points of an elliptic curve. In the cases when our curve has a so called flex point (intersection number with the tangent is $3$), the classical abelian group and the elliptic group carry the same information. However, if our curve does not have such a point (which often happens over $\mathbb{Q}$), the abelian group is not enough to recover the elliptic group.

The aim of this paper is to study this algebraic structure in more detail, its connections to abelian groups and at the very end even introduce the notion of an elliptic ring (a monoid object in the category of elliptic groups).
\end{abstract}

\section{Introduction}\label{sec1}

Let $\Ec$ be an elliptic curve defined over a field $K$. It is well known that the set $\Ec(K)$ of $K$-points has a group structure, denoted by $+$. The definition of $+$ is based on a binary operation $\ast$, called chord-tangent law \cite{dale}, and a chosen point $0\in \Ec(K)$. The chord-tangent law satisfies some identities spelled out in \cite{eltate}. This gives rise to an algebraic structure which we call an elliptic group (see Definition \ref{21}). Any elliptic group, together with a chosen point $o$, gives rise to an abelian group by  setting $a+b:=o\ast (a\ast b)$. It is well known that the choice of the distinguished point is not of paramount importance, since the corresponding groups are isomorphic. However, after passing from an elliptic group to the corresponding abelian group, we loose some information, since non-isomorphic elliptic groups could have isomorphic associated abelian groups. Such situations can be realised geometrically, see Example \ref{exa}. This motivates us to investigate the algebraic properties of elliptic groups.

We will show that elliptic groups are closely related to abelian groups. We describe free elliptic groups, automorphisms of various elliptic groups and we give a full list of finitely generated indecomposable elliptic groups. We further show that any finitely generated elliptic group is isomorphic to a product of indecomposable ones. In many cases, we give the construction of the coproduct of elliptic groups explicitly, which, unlike for abelian groups, differs from the product.

When I finished this work, I found that elliptic groups were studied by many authors under the name ''totally symmetric and entropic quasigroups'', see \cite{schwenk} and the literature cited therein. Results from Sections 2, 3 and 6 are known \cite{schwenk}, Corollary 4.3 is also proven in \cite{schwenk}. I included detailed account of these facts for the convenience of the reader and to have it all expressed in terms of elliptic groups. To the 

The content of this paper is organised as follows: In Section \ref{sec2}, we solve an exercise from the textbook \cite{eltate}, which is the basic motivation to introduce elliptic groups. Section \ref{sec3} establishes a relationship between abelian and elliptic groups. In Section \ref{sec4} we examine the morphisms and isomorphisms of elliptic groups. We proceed in Section \ref{sec5} to study the case of flex points. That is, when a point $o$ is given, such that $o\ast o=o$. In the following section, we equip the homset $\mor(S,T)$ with the structure of an elliptic curve. In Section \ref{sec7} we give a brief result about congruences. The proceeding section gives an explicit construction of free elliptic groups. Section \ref{sec9} deals with the study of finitely generated elliptic groups and the direct product decomposition of such objects. In Section \ref{sec10} we study coproducts and give several explicit constructions. We then study the automorphism groups in Section \ref{sec11} and in Section \ref{sec11}, we consider the tensor product. This allows us to define elliptic rings, which we study in more detail in Section \ref{sec13}, including the notion of primes and the fundamental theorem of arithmetic.

\section{An exercise of Silverman-Tate}\label{sec2}

Except for the name, the following definition goes back to \cite[p34]{eltate} Exercise 1.11. 
\begin{De}\label{21} An elliptic group is a set $S$ equipped with a binary operation $\ast$ satisfying the properties 
\begin{itemize}
	\item {\rm EG1)} $x\ast y=y\ast x$. 
	\item {\rm EG2)} $x\ast (x\ast y) =y$.
	\item {\rm EG3)} $x\ast (y\ast (z\ast w))= w\ast (y\ast (z\ast x))$.
\end{itemize}
\end{De}

\begin{Ex} Let $\Ec$ be an elliptic curve defined over a field $K$. Take two points $P,Q\in \Ec(K)$. The line $\mathcal{L}_{P,Q}$ passing through $P$ and $Q$ intersects $\Ec$ on a third (not necessarily distinct) point, denoted by $P\ast Q$. Thanks to \cite[Exercises 
1.10]{eltate}, one obtains an elliptic group in this way.
\end{Ex}

A morphism of elliptic groups $f:S\to T$ is a map, such that $f(x\ast y)=f(x)\ast f(y)$. The category of elliptic groups is denoted by $\Ell$.

The following is an exercise in \cite[Exercises 1.11]{eltate}.

\begin{Pro}\label{12} Let $S$ be a nonempty elliptic group. 

i) Fix an element $c$ in $S$ and define
$$x+_c y:= c \ast (x \ast y).$$
The pair $(S,+_c)$ is an abelian group with $c$ as the zero element. The inverse of an element $x$ is $ x \ast (c\ast c)$, thus $$-_c x=x \ast (c\ast c).$$

ii) If $d$ is another point, the map $\alpha:S\to S$, given by
$$f(x)=c\ast (x\ast d),$$
is a group isomorphism $(S,+_c)\to (S,+_d)$.
\end{Pro} 

\begin{proof} i) The commutativity of $\ast$ implies the commutativity of $+_c$. By EG2, we have
$$c+_dy=c \ast (c\ast y)=y.$$
Thus, $c$ is the zero element of $+_c$. We also have
$$x+_c (x\ast (c \ast c))= c\ast \left ( x\ast (x \ast (c\ast c))\right )=c \ast (c\ast c)
=c.$$
Hence, $x\ast (c \ast c)$ is the inverse of $x$. So far, we did not use condition EG3. We need it to prove the associativity of $+_c$. In fact we have 
$$x+_c(y+_cz) = c\ast \left ( x\ast( c\ast (y\ast z)) \right ).$$
On the other hand
$$(x+_cy)+_c z= c \ast ( (c\ast ( x\ast y))\ast z ).$$

So, for associativity, it suffice to observe that $x\ast( c\ast (y\ast z))= z \ast( c\ast (y\ast x))= (c\ast ( x\ast y))\ast z$, for which we use EG3 and then EG1.

ii) We have to show that
$$f(x+_c y)=f(x)+_d f(y).$$
To this end, observe that $f(x)=x+_cd$. Thus
$$f(x+_cy)=x+_cy+_cd.$$
To proceed further let us show the following cancellation law:
$$(x+_cd)+_dz=x+_cz.$$
We have
\begin{align*}
(x+_cd)+_dz=&d\ast \left ( (c\ast (x\ast d))\ast z\right )\\ =&d\ast (z\ast ((x\ast d)\ast c))\\
=& c\ast (z\ast ((x\ast d)\ast d)).
\end{align*}
Here we used EG3. The axioms EG1 and EG2 yield
$$ (x\ast d)\ast d=d \ast (d\ast x)=x.$$
Thus we obtain
$(x+_cd)+_dz=c\ast (z\ast x)=z+_cx,$
which is the cancellation law. Based on this, we compute
$$f(x)+_d f(y)=(x+_c d)+_d(y+_cd)=x+_c y+_c d.$$
Comparing our computations for $f(x)+_d f(y)$ and $f(x+_cy)$ above, we see that they are equal.
\end{proof}

Thus, we can talk about the group associated to an elliptic group. In what follows we write $x-_cy$ instead of $x+_c (-_c y)$, which equals to $c\ast (x\ast (y\ast (c\ast c)))$.

\begin{Ex}\label{exa} Let us consider the elliptic curves $\Ec_1$ and $\Ec_2$, defined over $\mathbb{F}_7$ by the equations
$$ x^3+2y^3=3$$
and 
$$ y^2=x^3+2$$
respectively. In both examples there are $9$ points: $$\Ec_1(\mathbb{F}_7)=\{(1,1), (2,1),(4,1),(1,2),(2,2),(4,2),(1,4),(2,4),(4,4)\}$$
and
$$\Ec_1(\mathbb{F}_7)=\{\infty, (0,3),(0,4),(3,1),(3,6),(5,1),(5,6),(6,1),(6,6)\},$$
where $\infty=(0:1:0)$ in the projective notation. Let us choose $\mathcal{O}=(1,1)\in \Ec_1(\mathbb{F}_7)$ and $\infty \in \Ec_2(\mathbb{F}_7)$. The corresponding abelian groups are isomorphic to $(\Z/3\Z)^2$ but as elliptic groups they are not isomorphic. In fact, in the second case we have $\infty \ast \infty=\infty $, while in the first case, there is not a single $p$ satisfying $P\ast P= P$. Such points are called flex points and we will discuss more about them in Section \ref{sec5}.
\end{Ex}

We will come back to this example in Example \ref{85}. We will use the following easy fact several times. 

\begin{Le} Let $f:S\to S'$ be a morphism of elliptic groups. The map $f$ is a group homomorphism $(S,+_c)\to (S',+_{f(c)})$, for any element $c\in S$.
\end{Le}

\section{Relations with abelian groups}\label{sec3}

\subsection{General remarks} According to Section \ref{sec1}, we can associate an abelian group to an elliptic group, by singling out an element. Our first result in this section shows that any abelian group can be obtained in such a manner. This is based on a simple construction which creates an elliptic group, starting with an abelian group with a chosen element. Moreover, any elliptic group can be obtained in this way. So, elliptic groups and abelian groups are closely related. As we will show, the isomorphism classes of abelian groups and elliptic groups are, however, not the same. More precisely, nonisomorphic elliptic groups could have isomorphic associated abelian groups and the aim of this section is to understand this phenomena more closely. 

Before we proceed further, let us observe that the theory of elliptic groups is an example of universal algebras and therefore, all known notions and results are applicable. In particular, we can talk not only about morphisms of elliptic groups, but also about elliptic subgroups, free elliptic groups, congruences and so on.

It is a consequence of the general facts from the theory of universal algebras, that the category $\Ell$ posses all limits and colimits and also, the fact that limits can be computed "componentwise". In particular, if $S$ and $T$ are elliptic groups, then $S\times T$ is also an elliptic group with the operation being defined by
$$(s,t){\ast }(s',t'):=(s\ast s',t\ast t').$$
The forgetful functor $\Ell\to Sets$ has the left adjoint. The value of this functor on a set $S$ is known as the free elliptic group generated by $S$. We will give an explicit construction of these objects in Section \ref{5free}.

\subsection{Reformulation of axioms of elliptic groups}\label{sec21}

Let $S$ be an elliptic group. For an element $a\in S$, denote by $\ell_a$ the map $S\to S$, given by $x\mapsto a\ast x$. Axiom EG2 can be equivalently stated as
$$\ell_a^2=id.$$
This can again be restated as:
$$\ell_a \ \ {\rm is \ a \ bijection \ and} \ \ \ell_a^{-1}=\ell_a.$$
Axiom EG1 can be formulated as follows: for all $a,b\in S$ one has
$$\ell_a(b)=\ell_b(a).$$
It follows that the family of maps $\ell_a$ are totally distinct, meaning that if $\ell_a$ and $\ell_b$ have the same value on at least one point, then they are equal. In fact, if $\ell_a(x)= \ell_b(x)$ for some $x$, then $\ell_x(a)=\ell_x(b)$ and since $\ell_x$ is a bijection, we have $a=b$. 

\subsection{Elliptic groups from abelian groups}

We start by constructing elliptical groups from abelian groups. 

\begin{Le}\label{ab-el} i) Let $A$ be an additively written abelian group. Fix an element $a$ and define
$$x\ast_a y:=a-x-y.$$
Then $(A,\ast_a)$ is an elliptic group, denoted by $_aA$. 

ii) If $f:A\to B$ is a group homomorphism and $a\in A$, then $f$ is also a morphism of elliptic groups $(A,\ast_a)\to (B, \ast_{f(a)}).$

iii) If $c$ is an element of $A$, the abelian group $(A,+_c)$, associated to the elliptic group $_aA$, is isomorphic to $A$. In fact, this group coincides with $A$ if $c=0$.
\end{Le}

\begin{proof} i) Commutativity of $\ast_a$ is obvious. Next, we have
$$x\ast _a(x\ast _a y)=a-x-(a-x-y)=a-x-a+x+y=y.$$
Finally, both $x\ast (y\ast (z\ast w))$ and $w\ast (y\ast (z\ast x)))$ are equal to $a+y-x-z-w$.

ii) This is straightforward to check.

iii) For any elements $x,y\in A$, we have
$$x+_cy=c\ast_a(x\ast_a y)=a-c-(a-x-y)=x+y-c.$$
We see that for $c=0$, the group law coincides with the original addition. For general $c$, the map $f:A\to A$, given by $f(x)=x+c$, is an isomorphism of abelian groups $A\to (A,+_c)$.
\end{proof}

\begin{Co} Any abelian group is a group associated to an elliptic group. 
\end{Co}

\begin{proof} This follows from part iii) of Lemma \ref{ab-el}.
\end{proof}

The next result shows that, up to isomorphisms, any elliptic group can be obtained as in Lemma \ref{ab-el}. This is well-known, see {\rm \cite{schwenk}}. 

\begin{Th}\label{2.3} Let $S$ be a nonempty elliptic group and $o\in S$. Denote by $A$ the abelian group $(S,+_o)$. For $c=o\ast o$, the elliptic group $_cA$ is the same as $S$.
\end{Th}

\begin{proof} To simplify the notation, we write $+$ instead of $+_o$. For an element $a$, define the map
$$\lambda_a:S\to S$$
by $\lambda_a(x)=a+x=o\ast(a\ast x)$. Thus, we have
$$\lambda_a=\ell_o\circ \ell_a,$$
which implies $\ell_a=\ell_o\circ \lambda_a$. Since $\ell_a^2=id$, we obtain
$$\ell_o\lambda_a\ell_o \lambda_a=id$$
or
$$\lambda_a\ell_o\lambda_{a}=\ell_o.$$
It follows that for all $x\in S$, we have
$$a+\ell_o(a+x)=\ell_o(x).$$
Denote $$c=\ell_o(o)=o\ast o.$$
Taking $x=c$, we obtain
$$a+\ell_o(a+c)=\ell_o(c)=o\ast (o\ast o)=o.$$
Since $o$ is the zero element for $+$, we see that
$$\ell_a(a+c)=-a.$$
Since $a$ was an arbitrary element we obtain 
$$\ell_o(x)=c-x.$$
Hence
$$\ell_a(x)=\ell_o\lambda_a(x)=\ell_0(a+x)=c-a-x.$$
Thus
$$a\ast x =c-a-x,$$
proving that the operations in $S$ and $_cA$ are the same. 
\end{proof}

\section{Morphisms of elliptic groups}\label{sec4}

If $S,T\in { \Ell}$, the set of morphism from $S$ to $T$ is denoted by $\mor(S,T)$. We use the notation $\Hom$ for the set of group homomorphisms.

\subsection{Morphism of elliptic groups and affine maps of abelain groups}
Let $A$ and $B$ be abelian groups. Recall that a map $f:A\to B$ is called \emph{affine} if the map $f_1:A\to B$ is a group homomorphism, where $f_1(x)=f(x)-f(0)$. Clearly any affine map $f:A\to B$ is of the form $f=f_0+f_1$, where $f_0$ is a constant map and $f_1$ is a group homomorphism. We say that $f_0$ is the \emph{constant} part of $f$ and $f_1$ is the \emph{linear part} of $f$.

\begin{Pro}\label{mor_el} Let $A$ and $B$ be abelian groups and $a\in A$, $b\in B$. Any morphism of elliptic groups $f:\, _aA\to\, _bB$ is affine as a map $A\to B$ 
and hence, of the form
$$f(x)=f_0+f_1(x), \ {\rm for\ all}\ x\in A,$$
where $f_0\in B$ and $f_1:A\to B$ is a group homomorphism.
Moreover, these quantities satisfy the compatibility condition:
\begin{equation}\label{3el}
3f_0+f_1(a)=b.\end{equation}
Conversely, if $f_1:A\to B$ is group homomorphism and $f_0\in B$ is an element for which the compatibility condition (\ref{3el}) holds, then
$$f(x)=f_0+f_1(x)$$
defines a morphism of elliptic groups $f:\, _aA\to\, _bB$.
\end{Pro}

\begin{proof} A map $f:A\to B$ is a morphism $_aA\to \, _bB$ if
$$f(x\ast_a y )=f(x)\ast_b f(y)$$
holds. In our circumstances, this means
\begin{equation}\label{morp}
f(a-x-y)=b-f(x)-f(y).
\end{equation}
Denote by $f_0$ the value of $f$ at zero, that is $f_0:=f(0)$. We also set $f_1(x)=f(x)-f_0$. Then $f(x)=f_0+f_1(x)$ and $f_1(0)=0$. Based on this, we can rewrite condition (\ref{morp}) as
$$3f_0+f_1(a-x-y)=b-f_1(x)-f_1(y).$$
By putting $x=y=0$, we obtain 
$$3f_0+f_1(a)=b.$$
We want to show that $f_1$ is a group homomorphism, which would imply that $f$ is affine. Replacing $b$ by $3f_0+f_1(a)$, we obtain
\begin{equation}\label{meore}
f_1(a-x-y)=f_1(a)-f_1(x)-f_1(y).
\end{equation}
Putting $y=0$ yields $f_1(a-x)=f_1(a)-f_1(x)$ for all $x$. Finally, replacing $x$ with $x+y$ and comparing it with equation (\ref{meore}) proves that $f$ is a group homomorphism.

Conversely, let $f=f_0+f_1$ be an affine map with $3f_0+f_1(a)=b.$ We have
$$f(x\ast_a y)=f(a-x-y)=f_0+f_1(a-x-y)=f_0+f_1(a)-f(x)-f(y).$$
On the other hand
$$f(x)\ast_b f(y)=b-f(x)-f(y)=b-2f_0-f_1(x)-f_1(y)= f_0+f_1(a)-f_1(x)-f_1(y).$$
Comparing these expressions, we see that $f(x\ast_a y )=(x)\ast_b f(y)$ and the result follows.
\end{proof}

\subsection{Two consequences} As we shall soon see, it might happen that there is no morphism between elliptic groups. The following gives a criteria when this happens. It also gives a criteria when $_aA$ and $_bB$ are isomorphic.

\begin{Co}\label{homiso} Let $A$ and $B$ be abelian groups and $a\in A$, $b\in B$. Then
$$\mor(_aA,\, _bB)\not = \emptyset \ \textnormal{ if and only if} \ \
b\in 3B+\mathcal{H}(a),$$
where $\mathcal{H}(a)\subset B$ is the set of all elements of the form 
$f(a)$, with $f$ running through the set of all group homomorphisms $A\to B$. Moreover, $_aA$ and $_bB$ are isomorphic if and only if
$$b\in 3B+\mathcal{I}(a).$$
Here, $\mathcal{I}(a)$ is the set of all elements of the form $f(a)$, where $f$ runs through the set of all group isomorphisms $A\to B$.
\end{Co}

\begin{proof} Assume there exist a morphism $f:\,_aA\to\, _bB$. We have seen that $f=f_0+f_1$, where $f_0$ is a constant map and $f_1$ is a group homomorphism. We have also seen that $3f_0+f_1(a)=b$. Thus, $b\in 3B+\mathcal{H}(a)$. Conversely, assume $b=3f_0+f_1(a)$ for some $f_1\in \Hom(A,B)$. We put $f=f_0+f_1$. Then $f$ is a morphism $_aA\to\, _bB$, thanks to Proposition \ref{mor_el}. If one can choose $f_1$ to be an isomorphism of abelian groups, then $f=f_0+f_1$ will also be bijective and hence, an isomorphism of elliptic groups.
\end{proof}

\begin{Co}\label{3invertible}
Assume the multiplication by $3$ is a surjective endomorphism of an abelian group $A$. Then for any $a \in A$, the elliptic groups $_aA$ and $_0A$ are isomorphic. 
\end{Co}

\begin{proof} This is a direct consequence of Corollary \ref{homiso}, as $3A=A$ and $a=\id(a)\in$ $\mathcal{I}(a)$.
\end{proof}

\section{Elliptic groups with a flex point}\label{sec5}

An element $o$ of an elliptic group $S$ is called a \emph{flex} point if $o\ast o=o$. For example, $0$ is a flex point of $_0A$, where $A$ is an abelian group $A$. It is clear that any homomorphism of elliptic groups sends flex points to flex points.

The category $\Ell$ has a terminal object ${\bf 0}$. As a set, it has the single element $o$, which is a flex point. For any elliptic group $S$, the unique morphism $S_!: S\to {\bf 0}$ has a section if and only if $S$ has a flex point.

\begin{Le} Let $S$ be an elliptic group and $c\in S$. Denote by $+_c$ the corresponding group structure. Then $c$ is a flex point of $S$ if and only if the inverse of any element $x\in S$ is $x\ast c$.
\end{Le}

\begin{proof} The inverse of $x$ is $x\ast (c\ast c)$ by Proposition 1.2 i), which implies the \emph{if} part. Conversely, assume $x\ast (c\ast c)=x\ast c$ for all $x$. It follows that $\ell_c=\ell_{c\ast c}$ and hence, $c\ast c = c$, thanks to Section \ref{sec21}.
\end{proof}

\begin{Le}\label{62} Let $A$ be an abelian group and $c\in A$. Then $_cA$ has a flex point if and only if $c\in 3A$. If this condition holds, the set of all flex points is a torsor over the group $$Ann_3(A)=\{a\in A\ |\ 3a=0\}.$$
\end{Le}

\begin{proof} Let $o$ be a flex point of $_cA$. We have
$$o=o\ast o=c-o-o,$$
hence $c=3o\in 3A$. If $t\in Ann_3(A)$, then $3t=0$. Hence
$$(o+t)\ast (o+t)=c-(o+t)-(o+t)=c-2o-2t=3o+3t-2o-2t=o+t.$$
This shows that $o+t$ is flex. Conversely, assume $o'$ is also a flex point. Denote $t'=o'-o\in A$. We have $o'=o+t'$. Thus, $c=3o'=3o+3t'=c+3t'$ and $3t'=0$.
\end{proof}

\begin{Le} i) An elliptic group $S$ is isomorphic to an elliptic group of the form $_0A$ if and only if $S$ has a flex point.

ii) Denote by $\Ell^{Flex}$ the category of elliptic groups with chosen flex points and morphisms preserving those points. The functor
$${\bf Ab\to Ell_{Gr}^{Flex}},\ \ \ A\mapsto \, _0A,$$
is an equivalence of categories, where ${\bf Ab}$ is, as usual, the category of abelian groups.
\end{Le} 

\begin{proof} i) This follows from Corollary \ref{3invertible}.

ii) Let $S$ be an elliptic group with a fixed flex point $o$. By Theorem 2.3, $_0A$ and $S$ are the same elliptic groups, where $A=(S,+_o)$ is the group corresponding to $S$. Hence, the above functor is essentially surjective. It is obviously faithful. It only remains to show that it is full. Let $f:S\to T$ be a morphism of elliptic groups. By assumption, we have a chosen flex $o\in S$ and $f(o)$ is the chosen flex in $T$. As we already said, $S=\,_0A, T\, =_0B$, where $A=(S,+_o)$ and $B=(T,+_{f(o)})$. The map $f$, considered as a map $A\to B$, sends $0$ to $0$. As such, the constant part of $f$ vanishes and thus, it is a homomorphism of abelian groups, proving the result.
\end{proof}

\begin{Le} For any elliptic group $S$, there is an injective morphism of elliptic groups $\kappa: S\to T$, where $T$ has a flex point.
\end{Le}

\begin{proof} We can assume $S=\, _cA$ for an abelian group $A$. We can embed $A$ into a divisible group (in fact, 3-divisible is enough) $B$. Then $_cB$ has a flex, thanks to Lemma \ref{62}, so we can take $T=\, _cA$.
\end{proof}

The above lemma can be compared with the well-known fact that any elliptic curve over an algebraically closed field has a flex point, see \cite[Ch.2. Proposition 4.8]{dale}. However, over non-algebraically closed fields, elliptic curves might not have a flex point and thus, their corresponding elliptic groups will be not of the form $_0A$. For example, the flex points of the elliptic
curve
$$x^3+2y^3-3z^3=0$$
are
$$(0, \sqrt[3]{3}, \sqrt[3]{2}), (\sqrt[3]{3},0,1) \ \ \textnormal{ and }\ \ (\sqrt[3]{2},-1,0).$$
Hence, this elliptic curve has no rational flex, see also Example \ref{exa}.

\begin{Ex}\label{85} Let us once again consider the elliptic curve from Example \ref{exa}. Since $\infty$ is a flex point of the elliptic group $\infty\in\Ec_2(\mathbb{F}_7)$, it follows that, as an elliptic group, it is isomorphic to $_0(\Z/3\Z)^2$.

If, on the other hand, we choose $\mathcal{O}=(1,1)\in \Ec_1(\mathbb{F}_7)$ as the zero element, the points $(1,2)$, $(2,4)$ become the generators for the group 
$$(\Ec_1(\mathbb{F}_7),+_{\mathcal{O}})\cong (\Z/3\Z)^2.$$
Since $\mathcal{O} \ast \mathcal{O}=(2,4)$, it follows from Theorem \ref{2.3}, that one has an isomorphisms of elliptic groups
$$\Ec_1(\mathbb{F}_7)\cong \, _{(1,0)}(\Z/3\Z)^2.$$
\end{Ex}

\section{$\mor(S,T)$ as an elliptic group}\label{sec6}

\begin{Le} Let $S$ and $T$ be elliptic groups. Assume $f,g:S\to T$ are homomorphisms of elliptic groups. Define $f\ast g:S\to T$ to be the map given by
$$(f\ast g)(s)=f(s)\ast g(s).$$
Then $f\ast g$ is also a homomorphism of elliptic groups.
\end{Le}

\begin{proof}
We can and we will assume that $S=\, _aA$ and $T=\, _bB$, for abelian groups $A$ and $B$, with $a\in A$, $b\in B$. To check that $h=f\ast g$ is a morphism of elliptic groups, it is enough to show that
$f\ast g$ is an affine map from the group $A$ to $B$, satisfying the compatibility condition:
$$b=3h_0+h_1(a).$$

To show this assertion, write $f(s)=f_0+f_1(s)$ and $g(s)=g_0+g_1(s)$, where $f_0,g_0\in B$ and $f_1,g_1\in \Hom(A,B)$. It follows that
$$h(s)=f(s)\ast g(s)=b-f_0-f_1(s)-g_0-g_1(s).$$
Thus, $h$ is affine with
$$h_0=b-f_0-g_0$$
and
$$h_1(s)=-(f_1+g_1).$$
Now, we can compute
$$3h_0+h_1(a)=3b-3f_0-3g_0-f_1(a)-g_1(a)=3b-(b-f_1(a))-(b-g_1(a))-f_1(a)-g_1(a)=b$$
and the Lemma follows. 
\end{proof}

\begin{Co} For any elliptic groups $S,T$, the set $\mor(S,T)$ of all homomorphisms of elliptic groups $S\to T$ is again an elliptic group.
\end{Co}

We will henceforth always equip the set $\mor(S,T)$ with this structure.

\subsection{Examples}

\begin{Le}\label{mor0b} Let $A$ and $B$ be abelian groups and $b\in B$. One has an isomorphism of elliptic groups
$$\mor(_0A,\,_bB)\cong
\begin{cases}\emptyset, \ \ {\rm if} \ b\not \in 3B\\
						  _0\Hom(\Z/3\Z\oplus A,B), \ \ {\rm if} \ b\in 3B.\end{cases}$$
The second isomorphism is natural with respect to $A$.
\end{Le}

\begin{proof} The first part follows directly from Corollary \ref{homiso}. Assume $b=3f_0$ for some $f_0\in B$. We observe that
$$\Hom(\Z/3\Z\oplus A,B)\cong \Hom(\Z/3\Z,B)\oplus \Hom(A,B)\cong {\sf Ann}_3(B)\oplus \Hom(A,B),$$
where ${\sf Ann}_3(B)=\{t\in B| 3t=0\}.$ Define the map
$$\phi: {\sf Ann}_3(B)\oplus \Hom(A,B)\to \maps(A,B)$$
by
$$\phi(t,h)(a):=f_0+t+h(a),$$
where $t\in B$ is element, such that $3t=0$ and $h\in\Hom(A,B)$. Obviously, $\phi(t,h):A\to B$ is an affine map with constant term $f_0+t$. Since
$$3(f_0+t)+h(0)=3f_0=b,$$
the compatibility condition \ref{3el} holds and thus, $\phi(t,h)\in\mor(_0A,\, _bB)$. Hence, $\phi$ can be considered as a map
$$\phi: {\sf Ann}_3(B)\oplus \Hom(A,B)\to \mor(_0A,\, _bB).$$
We have
\begin{align*}
\phi(-t-t',-h-h')(a)&=f_0-t-t'-h(a)-h'(a)\\&=
b-(f_0+t+h(a))-(f_0+t'+h'(a))\\
&= \phi(t,h)(a)\ast \phi(t',h')(a).
\end{align*}
Thus, $\phi:\,_0\Hom(\Z/3\Z\oplus A,B)\to \mor(_0A, \, _bB)$ is a morphism of elliptic groups. It suffices to show that it is a bijection. Take any morphism $g:\, _0A\to\, _bB$. By Proposition \ref{mor_el}, it has a form
$g=g_0+g_1$, where $3g_0=b$ and $g_1:A\to B$. By putting $h=g$ and $t=f_0-g_0$, we see that $g\mapsto (t,h)$ is the inverse of the map $\phi$.
 \end{proof}
 
\begin{Le}\label{mor1b} Let $B$ be an abelian group and $b\in B$.

i) One has an isomorphism of elliptic groups
 $$\mor(_1\Z,\,_bB)\cong \, _bB.$$

ii) Let $n=3^k$. There exists an isomorphism of elliptic groups
$$\mor(_1(\Z/n\Z),\,_bB)\cong \begin{cases}\emptyset, \ \ {\rm if} \ 3^kb\not \in 3^{k+1}B\\
						 _bB, \ \ {\rm if} \ 3^kb\in 3^{k+1}B. \end{cases}	$$
\end{Le}

\begin{proof} i) Any affine map $f:\Z\to B$ has the form $f(x)=f_0+xc$, where $f_0,c\in B$. The compatibility condition (\ref{3el}) reads
$$3f_0+c=b,$$
which shows that $f_0$ can be taken arbitrary and then $c=b-3f_0$ is uniquely determined. Thus, $f\mapsto f_0$ is an isomorphism.

ii) Any affine map $f:\Z/n\to B$ has the form $f(x)=f_0+xc$, where $f_0,c\in B$ and $nc=0$. Since $c=b-3f_0$, as argued above, the condition $nc=0$ implies $3^kb=3^{k+1}f_0$, and the result follows.
\end{proof}

\section{Congruences of elliptic groups}\label{sec7}

We consider congruences in elliptic groups next. According to the theory of universal algebras, an equivalence relation $\sim$ on an elliptic group $S$ is a \emph{congruence} if and only if $x\sim y$ and $u\sim v$ implies $x\ast u\sim y\ast v$. In this case, the quotient set $S/\sim$ has a unique elliptic group structure for which the canonical map $S\to S/\sim$ is a homomorphism. 

\begin{Le}\label{cong_sub} i) Let $S$ be a nonempty elliptic group and $\sim$ be a congruence on $S$. For any $c\in S$, the set
$$K_c=\{x\in S| x\sim c\}$$
is a subgroup of the group $(S,+_c)$.

ii) Conversely, if $K$ is a subgroup of $(S,+_c)$, then 
$$(x\sim y) \Longleftrightarrow (x+_c(-_c y)\in K)$$
is a congruence on $S$, where $-_cy=y\ast (c\ast c)$ is the inverse of $y$ in the group $(S,+_c)$.

\end{Le}

\begin{proof} i) Assume $x\sim y$ and $u\sim v$. We have
$$x+_cu=c\ast (x\ast u)\cong c\ast (y\ast y)=y+_cv.$$
Thus, $\sim$ is also a congruence in the group $(S,+_c)$ and hence, the result follows from the well-know correspondence between subgroups and congruences of abelian groups.

ii) The same correspondence shows that $\sim$ is a congruence in an abelian group $(S,+_c)$.
It remains to show that if $x\sim y$ and $u\sim v$, then $x\ast u\sim y\ast v$. 
Thanks to Theorem \ref{2.3}, we can assume that $S=\, _aA$. In this case, we have
$$x+_cy=a\ast (x\ast y)=a-c-(a-x-y)=x+y-a.$$ 
Thus, $K$ satisfies the following 
properties: a) $c\in K$, b) if $x,y\in K$, then $x+y-c\in K$ and c) if $x\in K$, then $2c-x\in K$. By definition, $x\sim y$ if and only if $$K\ni x-_cy=x+_c(2c-y)=x-y-c.$$
Similarly, $u-v-c\in K$. Now we have
$$y\ast v-x\ast u-c=a-y-v-a+x+u-c=(x-y-c)+(u-v-c)+c\in K,$$
thanks to the property b). 
\end{proof}

\begin{Co} The constructions in Lemma \ref{cong_sub} establish a one-to-one correspondence between congruences on $S$ and subgroups of $(S,+_c)$.
\end{Co}

\section{Free elliptic groups}\label{5free}\label{sec8}

\subsection{The description}

\subsubsection{The finitely generated case}

The first part of Lemma \ref{mor1b} shows that $_1\Z$ is a free elliptic group with one generator $0$. We now describe other free objects. As usual, let $e_1,\cdots,e_n$ denote the standard basis in the free abelian group $\Z^n$:
$$e_1=(1,\cdots,0) , \cdots, e_n=(0,\cdots,1).$$

\begin{Th} The elliptic group $_{e_1}\Z^n$ is a free object in $\Ell$, on the $n$-element set $\{0,e_2,\cdots,e_n \}$. In other words, $f\mapsto (f(0),f(e_2),\cdots, f(e_n))$ induces an isomorphism of elliptic groups:
$$\mor\left(\, _{e_1}\Z^n, \, _aA\right)\cong \, _{(a,\cdots,a)}A^n.$$
\end{Th}

\begin{proof}
Recall that the elliptic group structure on $_{e_1}\Z^n$ is given by 
$$(x_1,\cdots,x_n)\ast (y_1,\cdots,y_n)=(1-x_1-y_1,-x_2-y_2,\cdots, -x_n-y_n).$$
We have to show that for any elliptic group $S$ and elements $s_1,\cdots,s_n\in S$, there exists a unique morphism of elliptic groups $f:\, _{e_1}\Z^n\to S$, such that $f(0)=s_1,$ and $f(e_i)=s_i$ for all $2\leq i\leq n$. 
According to Theorem \ref{2.3}, we can assume $S=\, _cA$ for an abelian group $A$ and $c\in A$. 

{\bf Uniqueness}. Assume $f:\, _{e_1}\Z^n\to \, _cA$ is such a morphism. According to Proposition \ref{mor_el}, we have $f=f_0+f_1$, where $f_0$ is the constant function with value $f_0\in A$ and $f_1:\Z^n\to A$ is a group homomorphism. Since $f_0=f(0)=s_1$, is follows that $f_0$ is uniquely defined. On the other hand, $f_1$ is completely determined by the values $f_1(e_j)$ for $1\leq j\leq n$, as $f_1$ is a group homomorphism. Since $0\ast_{e_1}0=e_1$, we see that $f(e_1)=f(0)\ast_c f(0)=c-2s_1$.
Quite similarly, for any $2\leq i\leq n$, we have $f_1(e_i)=f(e_i)-f(0)=s_i-s_1$ and the uniqueness of $f$ is proven.

{\bf Existence}. The previous argument suggest that we consider
$$f(x_1,\cdots,x_n)=s_1+x_1(c-3s_1)+\sum_{i=2}^nx_i(s_i-s_1).$$
This formula defines a map $\Z^n\to A$. We first check that this map takes the expected values at $0,e_2,\cdots, e_n$. In fact, we have
$$f(0,\cdots,0)=s_1$$
and for any $2\leq i\leq n$, we have
$$f(e_i)=s_1+ 1(s_i-s_1)=s_i.$$
It thus only remains to be checked that this map is a morphism of elliptic groups. We have
\begin{align*}
f((x_1,\cdots,x_n)\ast_{e_1}(y_1,\cdots,y_n))&=f(1-x_1-y_1,-x_2-y_2,\cdots,-x_n-y_n)\\
&=s_1+(1-x_1-y_1)(c-3s_1)+\sum_{i=2}^n(-x_i-y_i)(s_i-s_1)\\
&=(c-2s_1)+(-x_1-y_1) (c-3s_1) +\sum_{i=2}^n(-x_i-y_i)(s_i-s_1).
\end{align*}
On the other hand
\begin{align*} f(x_1,\cdots,x_n)\ast_{c}f(y_1,\cdots,y_n)&=c-f(x_1,\cdots,x_n)-f(y_1,\cdots,y_n)\\
&=(c-2s_1)-(x_1+y_1)(c-3s_1)-\sum_{i=2}^n(x_i+y_i)(s_i-s_1).
\end{align*}
Comparing these expressions we see that
$f$ is a morphism of elliptic groups, implying the result.
\end{proof}

\subsubsection{The general case}

We have constructed the free elliptic group over a finite nonempty set. The free elliptic group on the empty set is just the empty set. The construction of a free elliptic group generated by an infinite set is quit similar to the finite case: If $X$ is any nonempty set, we let $A$ to be the free abelian group generated by $X$. Choose an element $a\in X$ and consider the elliptic group $_aA$. It is a free elliptic group on the set $0\cup X'$, where $X'=X\setminus a$.

\subsection{Finitely generated elliptic groups}

Let $S$ be an elliptic group. We say that a subset $X\subset S$ generates $S$ if $S$ is the minimal elliptic subgroup containing $X$. An elliptic group $S$ is called finitely generated if there is a finite subset $X\subset S$ which generates $S$. Equivalently, if there exists a surjective morphism of elliptic groups $F\to S$, where $F$ is a finitely generated free elliptic group.

\begin{Le} i) If $S$ is a nonempty and finitely generated elliptic group, the associated group is a finitely generated abelian group.

ii) If $A$ is a finitely generated abelian group, then for any $c\in A$, the elliptic group $_cA$ is a finitely generated elliptic group.
\end{Le}

\begin{proof} i) Since the associated group $(S,+_c)$ is independent of the chosen point $c\in S$, it suffices to show that $(S,+_c)$ is a finitely generated abelian group for at least one $c$. By assumption, there is a surjective morphism of elliptic groups $_{e_1}\Z^n\xto{f} S$, which gives rise to a surjective homomorphism of abelian groups $(\Z^n,+)\to (S,+_{f(0)})$. As such, $(S,+_{f(0)})$ is a finitely generated abelian group and part i) follows.

ii) Assume $a_1,\cdots, a_k$ are generators of an abelian group $A$. Consider the elliptic group $_{e_1}\Z^{k+1}$, which is the free elliptic group with generators $0,e_2,\cdots, e_{k+1}$. Thus, there is a homomorphism of elliptic groups $_{e_1}\Z^{k+1}\xto{f}\, _cA$, which satisfies $f(0)=0$ and $f(e_2)=a_1, \cdots, f(e_{k+1})=a_k$. The same $f$ can be seen as a homomorphism of abelian groups $\Z^{k+1}=(\,_{e_1}\Z^{k+1}, +_0)\to (_cA,+_0)=A$. Since $f$ is a group homomorphism whose image contains all the generators of the group $A$, it is surjective. It follows that $S$ is a finitely generated elliptic group.
\end{proof}

\subsection{A funny example}

We now turn to $_0\Z$. Recall that it is equipped with the operation $\ast_0$, defined by $x\ast_0y=-x-y$. Observe that the subsets $3\Z$, $1+3\Z$ and $2+3\Z$ are closed under $\ast_0$. So, they are elliptic 
subgroups of $_0\Z$, giving a partition by elliptic subgroups
$$_0\Z=(3\Z) \biguplus (1+3\Z )\biguplus (2+3\Z).$$ 
Here, $ \biguplus$ denotes the disjoint union, i.e. the coproduct in the category of sets.

The elliptic groups $1+3\mathbb{Z}$ and $2+3\mathbb{Z}$ are isomorphic to the elliptic group $_1\mathbb{Z}$. Hence, they are free with a single generator each. The corresponding isomorphisms
$$f:\, _1\mathbb{Z}\to 1+ 3\mathbb{Z},\ \ g:\, _1\mathbb{Z}\to 2+ 3\mathbb{Z}$$
are given by
$f(x)=3x+1$ and $g(x)=3x-1$. Thus, two out of the three elliptic subgroups in the partition of $_0\Z$ are free with one generator. The third one $3\Z$ is isomorphic to $_0\Z$, by sending $\Z\ni x\mapsto 3x\in \, 3\Z$. Hence, we can repeat this and obtain another partition
$$_0\Z= (1+3\Z) \biguplus (9\Z) \biguplus (3+9\Z) \biguplus (6+9\Z) \biguplus (2+3\Z),$$
where now four members are free with one generators and the fifth is isomorphic to the whole object. This can, of course, be repeated indefinitely.

\section{Indecomposable finitely generated elliptic groups}\label{sec9}

Let $\bf C$ be a category with products. An object $A$ is called \emph{indecomposable} if for any isomorphism $A\xto{(f_1,f_2)} A_1\times A_2$, either $A_1$ or $A_2$ is the terminal object. It is well-known, that in the category of finitely-generated abelian groups, the indecomposable objects are exactly $\mathbb{Z}$ and $\mathbb{Z}/p^k\mathbb{Z}$, where $p$ is prime and $k\geq 1$. Moreover, any finitely generated abelian group is isomorphic to the product of indecomposable ones and such a decomposition is unique up to the order of the factors. We can translate these classical results to elliptic groups.

\begin{Le}\label{51} Let $A$ and $B$ be groups.  For any elements $a\in A$ and $b\in B$, the elliptic groups $_{(a,b)}\, (A\oplus B)$ and $_aA\times\, _bB$ are isomorphic.
\end{Le}

\begin{proof} Both have the same underlying set. The operation in $_{(a,b)}\,
	(A\oplus B)$ is given by
$$(x,y)\ast_{(a,b)}(x',y')=(a,b)-(x,y)-(x',y')=(a-x-x',b-y-y')=(x\ast_a x',y\ast_by').$$
The last expression is the result of the operation in $_aA\times\, _bB$ and hence, the result follows.
\end{proof}

\begin{Le} There are exactly two elliptic groups, up to isomorphisms $_1 \mathbb{Z}$ and $_0 \mathbb{Z}$, whose associated groups are infinite cyclic groups. 
\end{Le}

\begin{proof} Since any group isomorphism $\mathbb{Z} \to \mathbb{Z} $ is given by either $x\mapsto x$ or $x\mapsto -x$, we see that in the notation of Corollary \ref{homiso}, we have $\mathcal{I}(a)=\{a,-a\}$. Hence, for integers $m,n$, the elliptic groups $_m\mathbb{Z}$ and $_n\mathbb{Z}$ are isomorphic if and only if $m\equiv n$ mod(3) or $m\equiv -n$ mod(3), thanks to Corollary \ref{homiso}. It follows that $_m\mathbb{Z}$ is isomorphic to either $_1\mathbb{Z}$ or $_0\mathbb{Z}$ and the result follows.
\end{proof}

\begin{Le} Let $p\not =3$. Up to isomorphisms, there is only one elliptic group $_0 \mathbb{Z}/p^k\Z$ whose associated group is the cyclic group $\mathbb{Z}/p^k\mathbb{Z}$.
\end{Le}

\begin{proof} Since $3$ is invertible, we can apply Corollary \ref{3invertible}.
\end{proof}

\begin{Le} Up to isomorphisms, there are exactly two elliptic groups $_1 \mathbb{Z}/3^k\mathbb{Z}$ and $_0 \mathbb{Z}/3^k\mathbb{Z}$ whose associated group is the cyclic group $\mathbb{Z}/3^k\mathbb{Z}$. 
\end{Le}

\begin{proof} Take $a\in \mathbb{Z}/3^k\mathbb{Z}=A$. If $a\in 3A$, then $_aA$ is isomorphic to $_0A$, thanks to Corollary \ref{homiso}. If $a\not \in 3A$, then $a$ is invertible in $A= \mathbb{Z}/3^k\mathbb{Z}$. So, $1=ab$ for some $b\in A$. Thus, $1= \mu_b(1)$, where $\mu_b:A\to A$ is the multiplication by $b$, which is an isomorphism of abelian groups Thus, in the notations of Corollary \ref{homiso}, we have $1\in \mathcal{I}(a)$. According to the same corollary, $_aA$ and $_1A$ are isomorphic. 
\end{proof}

The following fact was proven in \cite{schwenk}.

\begin{Th}\label{55} Any finitely generated elliptic group is isomorphic to the finite product of indecomposable elliptic groups. Moreover, the following is a full list of pairwise nonisomorphic, finitely generated, indecomposable elliptic groups
$$_0\Z,\, _1\Z,\, _0\Z/p^k\Z,\, _0\Z/3^k\Z,\,_1\Z/3^k,$$
where $p\not =3$ is a prime and $k\geq 1$ is an integer. 
\end{Th}

\begin{proof} Take a finitely generated elliptic group $S$. We know that the corresponding group is a finitely generated abelian group. Without loss of generality, we can assume that $S=\, _aA$, where $A$ is a finitely generated abelian group. Based on Lemma \ref{51}, we can conclude that $S$ is a finite product of elliptic groups of the form $_bB$, where $B$ is an indecomposable abelian group. Clearly, any such elliptic group will also be indecomposable as an elliptic group, because any decomposition of elliptic groups induces a decomposition of the corresponding groups. The rest just summarises previous lemmas.
\end{proof}

{\bf Remark}. Unlike abelian groups, the decomposition in Theorem \ref{55} is not unique. We have the following facts:

\begin{Le}\label{66} i) There is an isomorphism
$$f:\, (_1\Z) \,\times (_1\Z)=\, _{(1,1)}(\Z\oplus \Z)\to\, _{(1,0)}(\Z\oplus \Z)= (_1\Z) \times\, ( _0\Z),$$
given by $f(x,y)=(x,2x+y-1)$.

ii) For any integer $m$, the same map induces an isomorphism
$$\, (_1\Z) \,\times (_1\Z/3^m\Z)= \, _{(1,1)}(\Z \oplus \Z/3^m\Z)\to\, _{(1,0)}(\Z \oplus \Z/3^m\Z)=(_1\Z )\times\, (_0\Z/3^m\Z).$$

ii) For any integers $k\geq m$, the same map indu	ces an isomorphism
$$\, (_1\Z/3^k\Z) \,\times (_1\Z/3^m\Z)= \, _{(1,1)}(\Z/3^k\Z \oplus \Z/3^m\Z)\to\, _{(1,0)}(\Z/3^k\Z \oplus \Z/3^m\Z)=(_1\Z/3^k\Z )\times\, (_0\Z/3^m\Z).$$
\end{Le} 

The proofs can be done directly by hand and we omit them.

\begin{Co} Any finitely generated elliptic group is of one of the following forms:
$$_0A, \, _1\Z\times \, _0A \ \textnormal{or } \ _1\Z/3^k\Z \times \, _0A,$$ where $A$ is a finitely generated abelian group.
\end{Co}

\begin{proof} According to Theorem \ref{55}, any such object is a product of the above listed elliptic groups. Moreover, two copies of elliptic groups of the form $_1Z$, with $Z=\Z$ or $Z=\Z/3^k\Z$, can be replaced by a single one, according to Lemma \ref{66} and we are done.
\end{proof}

As we have seen, any elliptic group is embeddable into an elliptic group with a flex. We aim to prove a more precise statement for finitely generated elliptic groups.

\begin{Le} For any finitely generated elliptic group $S$, there is an injective morphism of elliptic groups $\kappa: S\to T$, where $T$ has a flex point and is also finitely generated.
\end{Le}

\begin{proof} Since $S$ is finitely generated, we can use Theorem \ref{55} to deduce the cases $S=\, _1\Z$ and $S=\, _1\Z/3^k\Z$, $k\geq 1$. In the first one, we have an injective morphism of elliptic groups $\kappa:\, _1\Z\to \, _0\Z$, where $\kappa(n)=1-3n$. In the second case, we take $\kappa:\, _1\Z/3^k\Z\to \, _0\Z/3^{k+1}\Z$, where $\kappa(x)=1-3x$.
 \end{proof}

\section{Coproduct}\label{sec10}

The goal of this section is to study the coproduct in the category of elliptic groups. In particular, our results give very explicit descriptions of the coproduct of finitely generated elliptic groups.

As usual, $\oplus$ denotes the direct sum of abelian groups, which is simultaneously the product and the coproduct in the category of abelian groups. For coproducts in the category of elliptic groups, we will use the symbol $\coprod$, while $\times$ will denote the product of elliptic groups. We further let $[a]_n$ (or simply $[a]$) be the class of an integer $a$ in $\Z/n\Z$.

As a reminder, a morphism $f:\, _aA\to \, _cC$ of elliptic groups is a map for which the map
$f_1:A\to C$, given by $f_1(x)=f(x)-f(0)$, is a group homomorphism and the following compatibility condition holds:
$$3f(0)+f_1(a)=c.$$
The last condition is equivalent to $2f(0)+f(a)=c$. 

\begin{Pro} Let $A$ and $B$ be abelian groups. There is an isomorphism of elliptic groups:
$$_0A\coprod \, _0B \cong \, _0(A\oplus B\oplus \Z/3\Z).$$	
More precisely, the diagram of elliptic groups
$$i:\, _0A \to \ _{0} (A\oplus B\oplus \Z/3\Z) \leftarrow \ _0B\, :j,$$
given by $i(x)=(x,0,0)$, $j(y)=(0,y,[1])$, is a coproduct diagram of elliptic groups. That is, for any elliptic group $E$ and for any morphisms of elliptic groups $f:\, _0A\to E$, $g:\, _0B\to E$, there exists a unique morphism of elliptic groups
$$h:\, _0(A\oplus B\oplus \Z/3\Z)\to E,$$
such that $f=h\circ i$ and $g=h\circ j$.
\end{Pro}

\begin{proof} Since $A\mapsto \, _0A$ is a functor $\bf Ab \to Ell_{Gr}$, the map $i$ is a morphism of elliptic groups. Since $[1]$ is of order $3$, one sees that the compatibility condition holds for $j$ and thus, $j$ is also a morphism. Take $E= \, _cC$ for an abelian group $C$ and an element $c\in C$. For $f$ and $g$ the compatibility conditions 
$$3f(0)=c=3g(0)$$
holds. Define the map $\tilde{h}: A\oplus B\oplus \Z\to C$ by
$$\tilde{h}(x,y,[n])= f(x)-nf(0)+g(y)+(n-1)g(0).$$
Since
$$\tilde{h}(x,y,[n+3])-\tilde{h}(x,y,[n])=-3f(0)+3g(0)=-c+c=0,$$
we see that $\tilde{h}$ yields a well-defined map
$$h: A\oplus B\oplus \Z/3\Z\to C.$$
Next,
$$3h(0,0,[0])=3(f(0)+g(0)-g(0))=3f(0)=c.$$
So, $h$ is a morphism of elliptic groups $\, _0(A\oplus B\oplus \Z/3\Z)\to \, _cC$. We have
$$h\circ i(x)=h(x,0,[0])=f(x)+g(0)-g(0)=f(x)$$
and
$$h\circ j(y)=h(0,y,[1])= f(0)-f(0)+g(y)=g(y).$$
We have checked all the conditions except for the uniqueness.

To this end, assume $p:, _0(A\oplus B\oplus \Z/3\Z)\to \, _cC$ is also a morphism, such that $p\circ i=f$ and d$p\circ j=g$. We have
$$p(0,0,0)=pi(0)=f(0).$$
So, $h$ and $p$ have the same constant terms. Hence $q=h-p$ is a group homomorphism.
Clearly $q\circ i=0$. So $q(x,0,0)=0$ for all $x$.
We also have
$$q\circ j(y)=h\circ j(y)-p\circ j(y)= g(y)-g(y)=0.$$
It follows that
$q(0,y,1)=0$ for all $y$. First, by putting $y=0$, we obtain that $q(0,0,[1])=0$ and hence,
$q(0,0,[n])=0$ for all $n$. Using the fact that $q$ is a group homomorphism, we obtain
$$q(0,y,0)=q((0,y,1)-(0,0,1))=q(0,y,[1])-q(0,0,[1])=0.$$
Thus, the restriction of the group homomorphism $q$ on each summand is zero and therefore $q$ itself is zero. This finishes the proof.
\end{proof}

\begin{Pro} Let $A$ and $B$ be abelian groups and $a$ an element of $A$. One has an isomorphism of elliptic groups:
$$(_1\Z\times \, _0B)\coprod\ _aA\cong \, 	_{(0,0,a)} (\Z\oplus B\oplus A).$$
More precisely, the diagram of elliptic groups
$$i:\, _1\Z \times \, _0B\to \ _{(0,0,a)} (\Z\oplus B\oplus A) \leftarrow \ _aA\, :j,$$
given by
$i(n,b)=(1-3n,b,na)$ and $j(x)=(0,0,x)$, is a coproduct diagram of elliptic groups. That is to say, for any elliptic group $E$ and any morphisms $f:\, _1\Z\times \, _0B\to E$ and $g: \, _aA\to E$, there exists a unique morphism $h:\, _{(0,0,a)} (\Z\oplus B\oplus A)\to E$, such that $f=h\circ i$ and $g=h\circ j$. 
\end{Pro}

\begin{proof} We first have to show that $i$ and $j$ define morphisms of elliptic groups. Since they obviously are affine maps, we need to check that they satisfy the compatibility condition. We have $2i(0,0)+ i(1,0)=2(1,0,0)+(-2,0,a)=(0,0,a)$. Thus, $i$ is a morphism. We also have $2j(0)+j(a)=(0,0,a)$ and hence, $j$ is also a morphism.

Without loss of generality, we can assume that $E=\, _cC$. Our next aim is to show that if $f$ and $g$ are morphisms, there exist a unique morphism $h$ with the above described properties.

For uniqueness, observe that any affine map $h:\Z\oplus B\oplus A\to C$ has the form
$$h(n,b,x)= u+nv+\alpha(x)+\beta(b),$$
where $(n,b,x)$ is any element from $\Z\oplus B\oplus A$ and $u,v\in C$ and $\alpha:A\to C$, $\beta:B\to C$ are group homomorphisms. The condition $h\circ j=g$ gives us $g(x)=h(0,0,x)=u+\alpha(x)$ for all $x\in A$. This implies that $u=g(0)$ and $\alpha(x)=g(x)-g(0)$. Thus, $u$ and $\alpha$ are uniquely defined. The condition $f=h\circ i$ gives
$$f(n,b)=h(1-3n,b,na)=u+(1-3n)v+\alpha(na)+\beta(b).$$ 
Putting $n=0, b=0$, we obtain $f(0,0)=u+v$. Thus, $v=f(0,0)-u=f(0,0)-g(0)$ is well-defined. We put $n=0$ to obtain $f(0,b)=u+v+\beta(b)$. It follows that $\beta(b)=f(0,b)-f(0,0)$ and hence, $\beta$ is also uniquely defined. 

For existence, we put
$$h(n,b,x)=g(0)+n(f(0,0)-g(0))+g(x)-g(0)+f(0,b)-f(0,0)$$
as the uniqueness proof suggests. By definition, $h$ is affine. We have
$$2h(0,0,0)+h(0,0,a)=2g(0)+g(0)+g(a)-g(0)=2g(0)+g(a)=c,$$
since $g$ is a morphism. It follows that $h$ is a morphism
$$h:\, _{(0,0,a)} (\Z\oplus B\oplus A)\to _cC.$$
Let us compare $h\circ j$ and $g$. We have
$$h\circ j(x)=h(0,0,x)=g(0)+g(0)-g(0)=g(x),$$
as required. We write
\begin{eqnarray*} h\circ i(n,b)&=&h(1-3n,b,na)\\
												  &=&g(0)+(1-3n)(f(0,0)-g(0))+g(na)-g(0)+f(0,b)-f(0,0)\\
												  &=& f(0,0)-3nf(0,0)+3ng(0)+n\alpha(a)+\beta(b)\\
												  &=&f(0,0)-3nf(0,0)+ nc +\beta(b).
\end{eqnarray*}
It will be helpful to write $f(n,b)=f(0,0)+nw+\beta(b)$, where $w\in C$ and $\beta:B\to C$ is a group homomorphism. The compatibility condition for $f$ gives $c=3f(0,0)+w$. Hence, we have
$$h\circ i(n,b)=f(0,0)+nw+\beta(b)=f(n,b)$$
and the proof is finished.
\end{proof}

\begin{Pro} Let $A$ and $B$ be abelian groups. There is an isomorphism of elliptic groups
$$(_1 \Z/3^k\Z \times \, _0B) \coprod \, _0A \cong \, _0(\Z/3^{k+1}\Z \oplus B \oplus A).$$	
More precisely, the diagram 
$$i:\, _1 \Z/3^k\Z \times \, _0B \to \ _0(\Z/3^{k+1}\Z\oplus B \oplus A) \leftarrow \ _0A\,:j,$$	
given by $i([a]_{3^k})=([1-3a]_{3^{k+1}},0)$, $j(x)=(0,0,x)$, is a coproduct diagram of elliptic groups. That is, for any elliptic group $E$ and for any morphisms of elliptic groups $f:\, _1 \Z/3^k\Z\times \, _0B\to E$, $g:\, _0A\to E$, there exists a unique morphism of elliptic groups
$$h:\, _0(\Z/3^{k+1}\Z \oplus A) \to E,$$
such that $f=h\circ i$ and $g=h\circ j$.
\end{Pro}

The proof of is quite similar to the previous case and we omit it.

\begin{Pro} Let $A$ and $B$ be abelian groups and $\ell\leq k$ be natural numbers. We have an isomorphism of elliptic groups
$$(_1 \Z/3^k\Z \, \times \, _0A)\coprod \, (_1 \Z/3^\ell\Z \, \times \, _0B)\cong \, _{(0,1,0,0)}(\Z/3^{k+1}\Z \oplus \Z/3^{\ell}\Z \oplus A\oplus B ).$$
More precisely, the diagram 
$$i:\, _1 \Z/3^k\Z \, \times \, _0A \to \ _{(0,1,0,0)}(\Z/3^{k+1}\Z \oplus \Z/3^{\ell}\Z \oplus A\oplus B ) \leftarrow \ _1 \Z/3^\ell\Z \, \times \, _0B \, :j,$$
given by $i([n]_{3^k},a)=([1-3n]_{3^{k+1}},[n]_{3^\ell},a,0)$, $j([m]_{3^\ell},b)=(0,[m]_{3^\ell},0,b)$ is a coproduct diagram of elliptic groups. This in term means that for any elliptic group $E$ and for any morphisms of elliptic groups $f:\,_1 \Z/3^k\Z \, \times \, _0A\to E$, $g:\, _1 \Z/3^\ell\Z \, \times \, _0B \to E$, there exists a unique morphism of elliptic groups
$$h:\, 	_{(0,1,0,0)}(\Z/3^{k+1}\Z \oplus \Z/3^{\ell}\Z\oplus A\oplus B) \ \to E,$$
such that $f=h\circ i$ and $g=h\circ j$.
\end{Pro}

\begin{proof} The map $j$ is a group homomorphism and $2j(0,0)+j(1,0)=(0,1,0,0)$. Hence, the compatibility condition holds and $j$ is a morphism of elliptic groups.

We next aim to show that $i$ is well-defined. This requires us to prove that $i([n]_{3^k},a)=i([n+3^k]_{3^k},a)$. This holds because
$$ ([1-3n]_{3^{k+1}},[n]_{3^\ell}, [a],0)= [1-3(n+3^k)]_{3^{k+1}},[n+3^k]_{3^\ell},a,0). $$
The condition $k\geq \ell$ is essential here. Since $i$ is well-defined, it is affine and the compatibility condition $2i(0,0)+i(1,0)=2(1,0,0,0)+(-2,1,0,0)=(0,1,0,0)$ holds. Thus, $i$ is a morphism.

Without loss of generality, we can assume that $E=\ _cC$. Our next aim is to show that if $f$ and $g$ are morphisms, there exist a unique morphism $h$ with the described properties.

For uniqueness, we observe that any affine map $h:\Z\oplus\Z \oplus A\oplus B\to C$ has the form
$$h(n,m,a,b)=u+nv+mw+\alpha(a)+\beta(b),$$
where $u,v,w\in C$ and $\alpha:A\to C$, $\beta:B\to C$ are group homomorphisms. The condition $h\circ j=g$ yields
$$u=h(0,0,0,0)=h\circ j(0,0)=g(0,0).$$
Quite similarly, we obtain
$$u+w=g(1,0), \ \ u+\beta(b)=g(0,b).$$
In the same way, the condition $h\circ i=f$ gives us 
$$h(1,0,0,0)=u+v, \ \ h(1,0,0,0)=u+v+\alpha(a).$$
As $u,v,\alpha,\beta$ are unique, $h$ too is uniquely defined.

For the existence, we first observe that since $f$ is affine, it can be written as 
$$f([n]_{3^k},a)= p+nq+\alpha(a),$$
where $p,q\in C$ and $\alpha:A\to C$ is a group homomorphism.
The compatibility condition says that
$3p+q=c$. Hence
$$f([n]_{3^k},a)= p+n(c-3p)+\alpha(a).$$
Since the LHS depends only on the residue class of $n$, the same must also be true for the RHS. Thus
$$3^{k}c=3^{k+1}p.$$
Quite similarly, for $g$ we have 
$$g([m],b)=r+m(c-3r)+\beta(b)$$
and $$ 3^{\ell}c=3^{\ell+1}r.$$
Here, $\beta:B\to C$ is a group homomorphism. Since $k\geq \ell$, we obtain
$$3^{k+1}(r-p)=0.$$
Define $h$ by 
$$h([n]_{3^{k+1}},[m]_{3^\ell},a,b)=r+n(p-r) +m(c-3r)+\alpha(a)+\beta(b).$$
The last equality implies that $h$ is a well-defined affine map
$$\Z/3^{k+1}\Z \oplus\Z/3^{\ell}\Z\oplus A\oplus B\to C.$$
We have
$$2h([0],[0],0,0)+h([0],[1],0,0)=2r+r+c-3r=c.$$
This shows that the compatibility condition also holds for $h$. We have
$$hj([m],b)=h([0],[m],0,b)=r+m(c-3r)+\beta(b)=g([m],b)$$
and
\begin{eqnarray*} hi([n],a)&=&h([1-3n],[n],a,0)\\
											 &=&r+(1-3n)(p-r)+n(c-3r)+\alpha(a)\\
											 &=&p-3np+nc+\alpha(a)\\
											 &=&p-3np+n(3p+q)+\alpha(a)\\
											 &=&p+nq+\alpha(a)=f([n],a).
\end{eqnarray*}
Thus $hi=f$ as required.
\end{proof}

\section{The group of automorphisms of an elliptic group}\label{sec11}

Let $S$ be an elliptic group. We let $\aute(S)$ be the group of automorphism of $S$ in the category ${\Ell}$. Denote by $\aut(A)$ the group of automorphisms of the group $A$. 

Recall that any morphism $F:\,_cA\to \, _cA$ is of the form $f=f_0+f_1$, where $f_0:A\to A$ is a constant map with value $f_0$ and $f_1:A\to A$ is a group homomorphism, such that
$$3f_0+f_1(c)=c.$$

\begin{Pro} Let $A$ be an abelian group and $c\in A$. For an elliptic group $S=\, _cA$, one has an exact sequence
$$0\to {\sf Ann}_3(A)\xto{\beta} \aute(\, _cA)\xto{\alpha} \aut(A)\xto{\eta} A/3A,$$
where ${\sf Ann}_3(A)=\{a\in A| 3a=0\}$ and the maps $\alpha,\beta,\eta$ are defined by
$$\beta(a)=a+1_A, \ \alpha(f_0+f_1)=f_1, \ \eta(g)=c-g(c)+3A.$$
Here, $g\in \aut(A)$, $a\in A$, $3a=0$ and $f=f_0+f_1$ is an automorphism of $_cA$.
 \end{Pro}

\begin{proof} An affine map $f=f_0+f_1$ is an isomorphism if and only if $f_1$ is an isomorphism. Hence, the map $\alpha$ is well defined. We have already observed that
$$3f_0+f_1(c)=c.$$
Since the kernel of $\alpha$ corresponds to the case $f_1=1_A$, we have $3f_0=0$ and the result follows. 
\end{proof}

\begin{Co} i) If $S=\, _0A$, the group $\aute(S)$ is isomorphic to the semidirect product ${\sf Ann}_3(A) \rtimes \aut(A)$, where the group $\aut(A)$ acts on ${\sf Ann}_3(A)$ by $(g,a)\mapsto g(a)$. 

ii) If $S$ is a free elliptic group with $n$ generators, then $\aute(S)$ is isomorphic to the subgroup of $GL_n(\Z)$, consisting of invertible integral matrices $(a_{ij})$, such that $a_{11} \equiv 1 ({\sf mod}\, 3) $ and $a_{i1} \equiv 0 ({\sf mod}\, 3)$ for all $2\leq i\leq n$.
\end{Co}

\begin{proof} i) In this case, $\eta=0$ and $g\mapsto 0+g$ gives a splitting of $\alpha$.

ii) In this case $A=\Z^n$. Hence, ${\sf Ann}_3(A)=0$, $\aut(A)=GL_n(\Z)$ and $\eta(a_{ij})=(1-a_{11},a_{21},\cdots,a_{n1})+3A$.
\end{proof}

Let $S$ be an elliptic group and $G$ be a group. The group homomorphisms $G\to \aute(S)$ can be considered as the representations of $G$ in an elliptic group $S$. Such objects arise in the following situation: Let $\Ec$ be an elliptic curve defined over a field $K$ and assume $K\subset L$ is a finite Galois extension with Galois group $G$. We obtain a natural representation of $G$ in $ \Ec(L)$. We might come back to this situation in the future.

\section{Elliptic tensor product}\label{sec12}

Let $R,S$ and $T$ be elliptic groups. A map $f:R\times S\to T$ is called a \emph{bihomomorphism} of elliptic groups if for any $r\in R$ and $s\in S$, the maps $S\to T, x\mapsto f(r,x)$ and $R\to T, y\mapsto f(y,s)$ are homomorphisms. Denote by ${\sf Bimor}(R,S;T)$ the set of such maps. It is an elliptic group by pointwise operation and one has the following isomorphism
$$\Phi:\mor(R,{\mor}(S,T))\cong {\sf Bimor}(R,S;T).$$
Just like in the classical case, it is given by
$$\Phi(h)(r,s):=h(r)(s),$$
where $h:R\to \mor(S,T)$ is homomorphism of elliptic groups.

Let $R$ and $S$ be elliptic groups. As for abelian groups, there is an universal bihomomorphism $u:R\times S\to R\et T$, which satisfies the following property: for any bihomomorphsm $f:R\times S\to T$, there exist a unique homomorphism $\bar{f}:R\et S\to T$, such that $f= \bar{f}\circ u$. In particular, we have
$$\mor(R,\mor(S,T))\cong \mor(R\et S,T)).$$
It follows from the properties of adjoint functors, that the functor $X\mapsto R\et X$ preserves all colimits, in particular coproducts. 

\begin{Le} For any abelian groups $A$ and $B$, one has
\begin{itemize}
	\item[(i)] $_1\Z\et A\cong A$,
	\item[(ii)] $_0A\et \, _0B\cong\, _0(\Z/3\Z\oplus A/3A\oplus B/3B\oplus A\t B)$.
\end{itemize}
\end{Le}

\section{Elliptic rings}\label{sec13}

Using the tensor product, we can define monoid objects in the category $(\Ell,\et)$ of elliptic groups. We now wish to study them in more detail.

\begin{De} An associative elliptic ring is a quadruple $(A,\et, \ep,e)$, where $(A, \et)$ is an elliptic group, $\ep$ is a bilinear operation on $A$ and $e\in A$, such that the following axioms hold 
\begin{itemize}
	\item[(i)] $(x\et y)\ep z=(x\ep z)\et (y\ep z)$,
	\item[(ii)] $x\ep (y\et z)=(x\ep y)\et (x \ep z)$,
	\item[(iii)] $x\ep (y\ep z)=(x\ep y)\ep z$,
	\item[(iv)] $x\ep e=x=e \ep x$.
\end{itemize}
If additionally $x\ep y=y\ep x$, then we say that $(A,\ep,e)$ is a commutative elliptic ring.
\end{De}

For example $(\Z, \et,\ep,0)$ is a commutative elliptic ring, where $x\et y=1-x-y$ and $x\ep y= x+y-3xy$. In fact, it is the initial object in the category of elliptic rings.

\subsection{Examples of elliptic rings}

i) The endomorphisms $(\mor(S,S),\ast, \circ, \id)$ of an elliptic group $S$ is an elliptic ring under the composition $\circ$, while $\ast$ is defined by   
$$(f\ast g)(s)=f(s) \ast g(s),$$
for any $f,g\in \mor (S,S)$.

As an example, take $S=\, _1\Z$. Recall that any endomorphism of $S$ is of the form $f_a(x)=a+(1-3a)x$, for a uniquely defined $a\in \Z$ (see Lemma \ref{mor1b}). One easily checks that
$$f_a\ast f_b=f_{1-a-b}, \ \ f_a\circ f_b=f_{a+b-3ab},$$
from which we see that $x\circ y=x+y-3xy$ defines an elliptic ring structure on $_1\Z$, whose elliptic unit is $0$.

Quite similarly, any  endomorphism of $_0\Z$ is of the form $g_a(x)=ax$, for a uniquely defined $a\in \Z$ (see Lemma \ref{mor1b}). It is, once again, readily checked that
$$g_a\ast g_b=g_{-a-b}, \ \ g_a\circ g_b=f_{ab}.$$
As such, $x\circ y=xy$ defines an elliptic ring structure on $_0\Z$, whose elliptic  unit is $1$.

More generally, we have the following facts.

\begin{Le} i) Let $R$ be a ring with unit. Define
$$r\ast s:=1-r-s \ \  {\rm and} \ \  r\circ s:=r+s-3rs.$$
Then $(R,\ast,\circ,0)$ is an elliptic ring, denoted by $Ell_1(R)$.

ii) Let $R$ be a ring with unit. Define
$$r\ast s:=-r-s \ \  {\rm and} \ \  r\circ s:=rs.$$
Then $(R,\ast,\circ,1)$ is an elliptic ring, denoted by $Ell_0(R)$.
\end{Le}

The proofs are left as exercises to the reader. We have just shown that $\mor (\, _1\Z, \, _1\Z)$ and $\mor (\, _0\Z, \, _0\Z)$ are isomorphic to  $Ell_1(\Z)$ and $Ell_0(\Z)$, respectively.

To compute $\mor (\, _{e_1}\Z^n, \, _{e_1}\Z^n)$, we first fix some notations. As usual, $A=(a_{ij})$ denotes a matrix and $v=(v_1,\cdot,v_n)^t$ denotes a column vector (here $t$ stands for the transpose). We set ${\sf Aff}_n(\Z)$ to be the affine version of the multiplicative monoid of $n\times n$ matrices: as a set ${\sf Aff}_n(\Z)$ is $\Z^n\times {\sf Mat}_n(\Z)$, with multiplication being given by
$$(v,A)\circ (w,B)=(v+Av,AB).$$
It has also an elliptic group structure, where $\ast$ is defined as in the product of $(n+1)$-copies of ${e_1}\Z^n$. For example, for $n=2$ we have
$$\left (\begin{pmatrix} u_1\\u_2\end{pmatrix}, \begin{pmatrix} a_{11} &a_{12}\\a_{21}&a_{22}\end{pmatrix}\right ) \ast \left (\begin{pmatrix} v_1\\v_2\end{pmatrix}, \begin{pmatrix} b_{11} &b_{12}\\b_{21}&b_{22}\end{pmatrix}\right )=
\left (\begin{pmatrix} 1-u_1-v_1\\ -u_2-v_2\end{pmatrix}, \begin{pmatrix} 1-a_{11}-b_{11} &1-a_{12}-b_{12}\\-a_{21}-b_{12}&-a_{22}-b_{22}\end{pmatrix}\right ).$$
An \emph{elliptic matrix} is a pair $(u,A)\in {\sf Aff}_n(\Z)$, for which
$$a_{11}=1-3u_1,\ {\rm and}\ \ a_{k1}=-3u_k, \ 2\leq k\leq n.$$
One can show that the set ${\sf Ell.Mat}(\Z)$ of all elliptic matrices forms an elliptic ring with respect to the above defined $\ast$ and $\circ$. This elliptic ring is isomorphic to $\mor(\, _{e_1}\Z^n \, _{e_1}\Z^n)$. This straightforward exercise is left to the reader.

\subsection{Arithmaetical properties of the elliptic ring $Ell_1(\Z)$}

Recall that $Ell_1(\Z)$ is an elliptic ring whose elements are integers and operations are defined by
$$a\ast b=1-a-b,\quad a\circ b=a+b-3ab.$$
The unit element is $e=0$. The elliptic ring $Ell_1(\Z)$ is in fact the initial object of the category of elliptic rings. Once can compare it to the initial object in the category of rings.

In this section, we aim to show that the fundamental theorem of arithmetic holds in $Ell_1(\Z)$ as well. For this, we investigate some arithmetical properties of $Ell_1(\Z)$, but first some terminology.

If $a$ and $b$ be are non-zero integers, we say that $b$ is $\circ$-divisible by $a$ if $b=a\circ c$ for an integer $c$. A non-zero integer $a$ is called $\circ$-\emph{prime}, if it has no non-trivial decomposition. That is, if $a=b\circ c$ implies $b=0$ or $c=0$.

Here are a few positive $\circ$-primes: $1,2,4,6,8,10,14,\cdots $, while $-2,-4,-6,-10, \cdots$ are negative $\circ$-primes. One can note that $12$ and $-8$ are not in the list.

The multiplicative monoid of $Ell_1(\Z)$ is obviously cancellative.  As we will see, it is, in fact, a free commutative monoid. More precisely, the following is true:

\begin{Th}\label{th1_20}  i) An integer $a$ is $\circ$-prime if and only if $p=|3a-1|$ is a prime number in the classical sense.

ii) Any nonzero integer can be decomposed as $a=q_1\circ \cdots \circ q_k$, where each $q_i$ is $\circ$-prime and such a decomposition is unique (up to reordering).
\end{Th}

The proof, given below, is based on the properties of the map $\Sigma:\Z\to \Z$, defined by $\Sigma(a)=1-3a$. One easily checks that it is injective and
$$\Sigma(a\circ b)=\Sigma(a)\Sigma(b).$$
Moreover, the image of this map is exactly the coset $1+3\Z$. In other words, it is the set of integers which are congruent to $1$ modulo $3$.

\begin{Le}\label{le133_20} If $\Sigma(a)=mn$, then either $m=\Sigma(b)$, $n=\Sigma(c)$  or $m=-\Sigma(b)$, $n=-\Sigma(c)$ for some $b,c\in \Z$.
\end{Le}

\begin{proof} Since $mn=\Sigma(a)\equiv 1\, {\rm ( mod \ 3)}$, it follows that either $m\equiv n \equiv 1\, {\rm ( mod \ 3)}$ or $m\equiv n \equiv 2\, {\rm ( mod \ 3)}$. In the first case, $m$ and $n$ are in the image of $\Sigma$. In the second one, $-m$ and $-n$ are in the image of $\Sigma$.
\end{proof}

\begin{proof}[Proof of Theorem \ref{th1_20}] i) Assume $p=|\Sigma(a)|$ is prime and $a=b\circ c$. Then  $p=|\Sigma(b)|\cdot |\Sigma(c)|$. Thus, $\Sigma(b)=\pm1$ or $\Sigma(c)=\pm 1$. So, $b=0$ or $c=0$ and $a$ is $\circ$-prime.

Conversely, assume $a$ is $\circ$-prime. If $\Sigma(a)=mn$, then either $m=\Sigma(b), n=\Sigma(c)$ or  $m=-\Sigma(b), n=-\Sigma(c)$, thanks to Lemma \ref{le133_20}. Hence, $\Sigma(a)=\Sigma(b\circ c)$ in both cases. Since $\Sigma$ is injective, we obtain $a=b\circ c$, which yields $b=0$ or $c=0$. Thus, $m=\pm 1$ or $n=\pm 1$. Hence $p$ is prime.

ii) For the existence, take any non-zero number $a$ and consider the prime decomposition of $\Sigma(a)$
$$\Sigma(a)=up_1\cdots p_{k}q_1\cdots q_n,$$
where $u=\pm 1$, $p_1,\cdots, p_k$ are usual primes which are congruent to $1$ modulo $3$, while $q_1,\cdots, q_n$ are usual primes which are congruent to $2$ modulo $3$. Such a decomposition exists and is unique, since $3$ does not divide $\Sigma(a)$. By Lemma \ref{le133_20}, we have
$$p_i=\Sigma(b_i), q_j=-\Sigma(c_j).$$
Since $\Sigma(a)$ is congruent to $1$ modulo $3$, it follows that $n$ is even. Therefore, we have
$$\Sigma(a)=\Sigma(b_1\circ \cdots \circ b_kc_1\circ \cdots \circ c_n)$$
and the injectivity of $\Sigma$ gives us $a=b_1\circ \cdots \circ b_kc_1\circ \cdots \circ c_n$. This is a $\circ$-prime decomposition, thank to part i).

For the uniqueness, observe that if $q_1\circ \cdots \circ q_k= s_1\circ \cdots \circ s_n$ are two decompositions by $\circ$-primes integers, we obtain
$$\Sigma(q_1)\cdots \Sigma(q_k)=\Sigma(s_1)\cdots \Sigma(s_n).$$
This implies $|\Sigma(q_1)|\cdots |\Sigma(q_k)|=|\Sigma(s_1)|\cdots |\Sigma(s_n)|$. By part i), each term is prime. It follows that $k=n$ and after reindexing $|\Sigma(q_i)|=|\Sigma(s_i)|$. This implies $\Sigma(q_i)=\pm \Sigma(s_i)$. By considering this equality modulo $3$, we see that the sign must be $+1$ and the uniqueness follows from the injectivity of $\Sigma$.
\end{proof}

As a corollary of the theorem, we see that the only odd number which is $\circ$-prime is $1$ and an even number $2k$ is $\circ$-prime if and only if $|6k-1|$ is prime. Based on this, we can give an ``elementary proof'' of the fact that there are infinitely many primes of the form $6t-1$ (a particular case of Dirchlet's famous theorem). By our theorem, this is equivalent to the fact that there are infinitely many $\circ$-primes. This fact can be proven in exactly the same way as Euclid did for usual primes, we simply replace the expression $p_1\cdots p_k+1$ in that argument by $p_1\circ \cdots \circ p_k\ast 0$ and use Lemma \ref{2le_18} and the fact that the equation $x\circ y=0$ has only the trivial solution $x=y=0$.

\begin{Le}\label{2le_18} Assume $a$ is a non-zero integer. If $b$ and $b\ast c$ are $\circ$-divisible by $a$, then $c$ is $\circ$-divisible by $a$.
\end{Le}

\begin{proof} If $b=a\circ x$ and  $b\ast c=a\circ y$, we obtain
$$c=b\ast (b\ast c)=(a\circ x)\ast (a\circ y)=a\circ (x\ast y),$$
hence the result.
\end{proof}

\end{document}